\documentclass[draft,11pt]{article}

\usepackage{amsfonts,amsmath,amsthm,amscd,amssymb,latexsym,cite,verbatim,texdraw,floatflt,caption2,pb-diagram}

\usepackage[mathscr]{eucal}
\newtheorem{theorem}{Theorem}[section]
\newtheorem{lemma}{Lemma}[section]

\newtheorem{corollary}{Corollary}[section]
\theoremstyle{definition}

\newcommand{\keywords}{\textbf{Keywords: }\medskip}
\newcommand{\subjclass}{\textbf{Mathematics Subject Classification (2010):}\medskip}
\renewcommand{\abstract}{\textbf{Abstract.}\medskip}

\begin{document}

%%%%%%%%%%%%%%%%%%%%%%%%%%%%%%%%%%%%%%%%%%%%%%%%%%%%%%%%%%%%%%%%%%%%%%%%%%%%%
%%%%%%%%%%%%%%%%%%%%%%%%%%%%%%%%%%%%%%%%%%%%%(\ref {3})%%%%%%%%%%%%%%%%%%%%%%
%%%%%%%%%%%%%%%%%%%%%%%%%%%%%%%%%%%%%%%%%%%%%%%%%%%%%%%%%%%%%%%%%%%%%%%%%%%%%

\title{Jackson-type inequalities and widths of functional classes
  in the Musielak-Orlicz type spaces\thanks{This work was supported in part  by the Kyrgyz-Turkish Manas
University (Bishkek / Kyrgyz Republic),  project No.~KTM\"{U}-BAP-2019.FBE.02, the Ministry of Education and Science of Ukraine in the framework of the fundamental research No.~0118U003390  and the Volkswagen Foundation (VolkswagenStiftung), program ``From Modeling and Analysis to Approximation''.}}

\author{F.\,G.~Abdullayev, S.\,O.~Chaichenko, M.~Imashkyzy, A.\,L.~Shidlich}

\date{}

\maketitle

%%%%%%%%%%%%%%%%%%%%%%%%%%%%%%%%%%%%%%%%%%%%%%%%%%%%%%%%%%%%%%%%%%%%%%%%%%%%%%%%%%%%%%%%%%%%
\begin{abstract}

In the Musielak-Orlicz type spaces ${\mathcal S}_{\bf M}$,   exact Jackson-type inequalities are obtained in terms of
best approximations of functions and  the averaged values of their generalized moduli of smoothness. The  values of %the
Kolmogorov, Bernstein, linear, and projective widths in ${\mathcal S}_{\bf M}$ are found for classes of periodic functions defined by certain conditions on the averaged values of the generalized moduli of smoothness.

\end{abstract}



%%%%%%%%%%%%%%%%%%%%%%%%%%%%%%%%%%%%%%%%%%%%%%%%%%%%%%%%%%%%%%%%%%%%%%%%%%%%%%%%%%%%%%%%%%%%%%%%%%%%%%%%%%%%%%%%%%%%%%%%%%%%%%%%

%%%%%%%%%%%%%%%%%%%%%%%%%%%%%%%%%%%%%%%%%%%%%%%%%%%%%%%%%%%%%%%%%%%%%%%%%%%%%%%%%%%%%%%%%%%%%%%%%%%%%%%%%%%%%%%%%%%%%%%%%%%%%%%%

\section{Introduction}

Let ${\bf M}=\{M_k(t)\}_{k\in {\mathbb Z}}$, $t\ge 0$, be a sequence of Orlicz functions. In other words, for every $k\in {\mathbb Z}$, the
function $M_k(t)$ is a nondecreasing convex function for which $M_k(0)=0$ and $M_k(t)\to \infty$ as $t\to \infty$.
The modular space (or  Musilak-Orlicz  type space) ${\mathcal S}_{\bf M}$  is the space of $2\pi$-periodic  complex-valued Lebe\-sgue summable functions $f$ ($f\in L$)   such that
the following quantity (which is also called the Luxemburg norm of $f$) is finite:
\begin{equation}\label{def_Lux_norm}
    \|{f}\|_{_{\scriptstyle  {\bf M}}}:=%\|{f}\|_{_{\scriptstyle {\mathcal S}_M}}=
    \|\{\widehat{f}(k)\}_{k\in {\mathbb Z}}\|_{_{\scriptstyle l_{\bf M}({\mathbb Z})}}:=
    \inf\bigg\{a>0:\  \sum\limits_{k\in\mathbb Z}  M_k(|{\widehat{f}(k)}|/{a})\le 1\bigg\},
\end{equation}
where $\widehat{f}(k):={[f]}\widehat{\ \ }(k)={(2\pi)^{-1}}\int_0^{2\pi}f(x) \mathrm{e}^{- \mathrm{i}kx}\mathrm{d}x$,
$k\in\mathbb Z$, are the Fourier coefficients of   the function
$f$.

%%%%%%%%%%%%%%%%%%%%%%%%%%%%%%%%%%%%%%%%%%%%%%%%%%%%%%%%%%%%%%%%%%%%%%%%%%%%%%%%%%%%%%%%%%%%%%%%%%%%%%%%%%%%%%%%%%%%%%%%%%
 The spaces ${\mathcal S}_{\bf M}$ defined in this way are Banach spaces.  Functional  spaces of this type have been studied by mathematicians since the 1930s (see, for example, the monographs
  \cite{Lindenstrauss-1977}, \cite{Musielak-1983}, \cite{Rao_Ren_2002}). If all functions $M_k$ are identical (namely, $M_k(t)\equiv M(t)$,
 $k\in {\mathbb Z}$), the spaces ${\mathcal S}_{\bf M}$ coincide with the ordinary Orlicz type spaces ${\mathcal S}_{M}$
 \cite{Chaichenko_Shidlich_Abdullayev_2019}. If $M_k(t)=\mu_k t^{p_k}$,  $p_k\ge 1$, $\mu_k\ge 0 $, then ${\mathcal S}_{\bf M}$ coincide with the weighted spaces
${\mathcal S}_{_{\scriptstyle  \mathbf p,\,\mu}}$ with variable exponents \cite{Abdullayev_Chaichenko_Imash_kyzy_Shidlich_2020}.
If all $M_k(t)=t^p$, $p\ge 1$, then the spaces ${\mathcal S}_{\bf M}$ are the  known
spaces ${\mathcal S}^{p}$ (see, for example, \cite[Ch.~11]{Stepanets_M2005}),
which in the case $p=2$ coincide with ordinary Lebesgue spaces ${\mathcal S}^{2}=L_2$.

%%%%%%%%%%%%%%%%%%%%%%%%%%%%%%%%%%%%%%%%%%%%%%%%%%%%%%%%%%%%%%%%%%%%%%\bigskip

In the paper, we study the approximative properties of the spaces ${\mathcal S}_{\bf M}$. In particular,
 exact Jackson-type inequalities in ${\mathcal S}_{\bf M}$ are obtained in terms of
the best approximations of functions and the averaged values of their generalized moduli of smoothness. The  values of %the
Kolmogorov, Bernstein, linear, and projective widths in the spaces ${\mathcal S}_{\bf M}$ are found for classes of periodic functions defined by certain conditions on the averaged values of the generalized moduli of smoothness.

Jackson-type (or Jackson-Stechkin-type) inequalities are inequalities  that estimate  the values of the best approximations
of functions via the value  of their modulus of continuity (smoothness) at a certain point.
The first exact Jackson-type inequality for the best uniform approximations of $2\pi$-periodic continuous functions  by trigonometric polynomials was obtained by Korneichuk \cite{Korneichuk_1962} in 1962. A similar result for the best uniform approximations of continuous functions given on the real axis by entire functions of the exponential type
was obtained by Dzyadyk in \cite{Dzyadyk_1975}. In 1967,   Chernykh \cite{Chernykh_1967_1, Chernykh_1967_2} proved two unimprovable inequalities for  $2\pi$-periodic functions from the Lebesgue spaces $L_2$. In \cite{Chernykh_1967_2},  it was shown in particular that the averaged values of the moduli of smoothness can be more effective for characterizing the structural and approximative properties of the functions $f$ than the moduli themselves. In  \cite{Taikov_1976, Taikov_1979}  (see also \cite[Ch. 4]{Pinkus_1985}), Taikov originated systematic investigations of the problem of exact inequalities that estimate  the values of the best approximations of functions via the averaged values of their  moduli of smoothness. He first considered the
functional classes  of  $2\pi$-periodic  functions  defined by certain conditions on the averaged values of their moduli of
smoothness and found the exact values of the widths of such classes  in the spaces  $L_2$.
Later, similar topic was studied by  numerous mathematicians in various  functional  spaces  (see, for example,
\cite{Stepanets_Serdyuk_2002, Vakarchuk_2004, Israflov_Guven_2006, Akgun_Israflov_2011, Akgun_Kokilashvili_2011, Sharapudinov_2013, Babenko_Konareva_2019, Akgun_2019, Abdullayev_Serdyuk_Shidlich_2021}, etc.).
For more detailed information on the results obtained in this direction and the corresponding references, see also \cite{Shabozov_Yusupov_2011, Vakarchuk_2012, Vakarchuk_2016, Babenko_Konareva_2019}.

%%%%%%%%%%%%%%%%%%%%%%%%%%%%%%%%%%%%%%%%%%%%%%%%%%%%%%%%%%%%%%%%%%%%%%\bigskip

\section{Preliminaries }

\subsection{Orlicz  norm}

In addition  to the Luxemburg norm (\ref{def_Lux_norm}) of the space $
{\mathcal S}_{\bf M}$, consider  the
Orlicz norm that is defined as follows. Let
 ${\bf {M^*}}=\{{M^*_k}(v)\}_{k\in {\mathbb Z}}$, $v\ge 0$, be the sequence of  functions defined by the relations
 \[
  %\begin{equation} \label{M_tilde}
    {M^*_k}(v):=\sup\{uv-M_k(u): ~u\ge 0\}, \quad k \in \mathbb{Z}.
 \]
 %\end{equation}
Consider the set $\Lambda{=}\Lambda({\bf {M^*}})$ of %all
sequences of positive numbers $\lambda=\{\lambda_k\}_{k\in \mathbb{Z}}$
such that  $\sum_{k\in \mathbb{Z}}{M^*_k}(\lambda_k){\le} 1$. For any function  $f\in {\mathcal S}_{\bf M}$, define its Orlicz norm by the equality
\begin{equation} \label{def-Orlicz-norm}
    \|f\|^\ast_{_{\scriptstyle  {\bf M}}}:=\|\{\widehat{f}(k)\}_{k\in {\mathbb Z}}\|_{_{\scriptstyle l_{\bf M}^*({\mathbb Z})}}:= \sup \Big\{ \sum\limits_{k \in \mathbb{Z}}
    \lambda_k|\widehat{f}(k) |: \quad  \lambda\in \Lambda({\bf {M^*}})\Big\}.
\end{equation}

Further, we will mainly use the Orlicz norm for functions $f \in {\mathcal S}_{\bf M}$. However, taking into account the following Lemma \ref{Lemma_3}, some corollaries can also be formulated from the results obtained when considering the Luxemburg norm.

\begin{lemma}
      \label{Lemma_3}
      For any function $f \in {\mathcal S}_{\bf M}$, the following relation holds:
      \begin{equation} \label{estim-for-norms}
      \| f\|_{_{\scriptstyle  {\bf M}}} \le \| f\|^*_{_{\scriptstyle  {\bf M}}}\le 2 \,\| f\|_{_{\scriptstyle  {\bf M}}}.
      \end{equation}
\end{lemma}

Relation (\ref{estim-for-norms}) follows from the similarly relation for corresponding norms in the modular Orlicz sequence  spaces (see, for example \cite[Ch. 4]{Lindenstrauss-1977}).

%%%%%%%%%%%%%%%%%%%%%%%%%%%%%%%%%%%%%%%%%%%%%%%%%%%%%%%%%%%%%%%%%%%%%%%%%%%%%%%%%%%%%%%%%%%%%%%%%%%%%%%%%%%%%

%%%%%%%%%%%%%%%%%%%%%%%%%%%%%%%%%%%%%%%%%%%%%%%%%%%%%%%%%%%%%%%%%%%%%%%%%%%%%%%%%%%%%%%%%%%%%%%%%%%%%%%%%%%%%%%%%%%%%%%%

\subsection{Generalized moduli of smoothness and their averaged values}

%%%%%%%%%%%%%%%%%%%%%%%%%%%%%%%%%%%%%%%%%%%%%%%%%%%%%%%%%%%%%%%%%%%%%%%%%%%%%%%%%%%%%%%%%%%%%%%%%%%%%%%%%%%%
Let $\omega_\alpha(f,\delta)$ be the modulus of smoothness of a function  $f \in {\mathcal S}_{\bf M}$ of order
$\alpha>0$, i.e.,
\begin{equation}\label{usual_modulus}
    \omega_\alpha(f,t)_{_{\scriptstyle  {\bf M}}}^*:=
    \sup\limits_{|h|\le t}\|\Delta_h^\alpha f\|_{_{\scriptstyle  {\bf M}}}^*=
     \sup\limits_{|h|\le t} \Big\|\sum\limits_{j=0}^\infty (-1)^j {\alpha \choose j} f(\cdot-jh)
     \Big\|_{_{\scriptstyle  {\bf M}}}^*,
\end{equation}
where ${\alpha \choose j}=\frac {\alpha(\alpha-1)\cdot\ldots\cdot(\alpha-j+1)}{j!}$ for $j \in \mathbb{N}$ and ${\alpha \choose j}=1$ for $j=0$. By the definition, for any ${k}\in {\mathbb Z}$, we have
\begin{equation}\label{modulus_difference_Fourier_Coeff}
|{[\Delta_h^\alpha f]}\widehat {\ \ }(k)|=|1-\mathrm{e}^{-\mathrm{i}kh}|^\alpha |\widehat{f}(k)|
=2^\frac \alpha 2 (1-\cos{kh})^\frac \alpha2 |\widehat{f}(k)|.
\end{equation}
%%%%%%%%%%%%%%%%%%%%%%%%%%%%%%%%%%%%%%%%%%%%%%%%%%%%%%%%%%%%%%%%%%%%%%%%%%%%%%%%%%%%%%%%%%%%%%%%%%%%%%%%%%%%%%%%%%%%%%%%%%%%%%%%%
Consider the set $\Phi$ of all continuous bounded nonnegative
pair functions $\varphi$ such that
$\varphi(0)=0$ and the Lebesgue measure of the set $\{t\in {\mathbb R}:\,\varphi(t)=0\}$ is equal to zero. For a fixed function $\varphi\in \Phi$, $h\in {\mathbb R}$ and for any $f \in {\mathcal S}_{\bf M}$, we denote by
$\{{[\Delta_h^\varphi f]}\widehat {\ \ }(k)\}_{k\in {\mathbb Z}}$ the sequence of numbers such that  for any $k\in {\mathbb Z}$
  \begin{equation}\label{modulus_generalize difference_Fourier_Coeff}
  {[\Delta_h^\varphi f]}\widehat {\ \ }(k)  =\varphi(kh)  \widehat{f}(k).
\end{equation}
%%%%%%%%%%%%%%%%%%%%%%%%%%%%%%%%%%%%%%%%%%%%%%%%%%%%%%%%%%%%%%%%%%%%%%%%%%%%%%%%%%%%%%%%%%%%%%%%%%%%%%%%%%%%%%%%%%%%%%%%%%
If there exists a function $\Delta_h^\varphi f\in L$ whose Fourier coefficients coincide with the numbers
${[\Delta_h^\varphi f]}\widehat {\ \ }(k)$, $k\in {\mathbb Z}$, then, as above,
the expression  $\|\Delta_h^\varphi f\|^*_{_{\scriptstyle  {\bf M}}}$ denotes  the  Orlicz norm of the function
 $\Delta_h^\varphi f$.
 %%%%%%%%%%%%%%%%%%%%%%%%%%%%%%%%%%%%%
 If such a function does not exist, then
 the  notation   $\|\Delta_h^\varphi f\|^*_{_{\scriptstyle  {\bf M}}}$ denotes the  norm  $\|\cdot\|_{_{\scriptstyle l_{\bf M}^*({\mathbb Z})}}$ of the sequence $\{{[\Delta_h^\varphi f]}\widehat {\ \ }(k)\}_{k\in {\mathbb Z}}$.

Similary to  \cite{Shapiro_1968}, \cite{Boman_Shapiro_1971}, \cite{Boman_1980},
define  the generalized  modulus of smoothness   of a function $f \in {\mathcal S}_{\bf M}$  by the equality
\begin{equation}\label{general_modulus}
    \omega_\varphi(f,t)_{_{\scriptstyle  {\bf M}}}^*=\sup\limits_{|h|\le t} \|\Delta_h^\varphi f\|_{_{\scriptstyle  {\bf M}}}^*.
 \end{equation}
It follows from  (\ref{modulus_difference_Fourier_Coeff}) that $\omega_\alpha(f,t)_{_{\scriptstyle  {\bf M}}}^*\!=\omega_\varphi(f,t)_{_{\scriptstyle  {\bf M}}}^*$  when
$\varphi(t)\!=\varphi_\alpha(t)=\!2^\frac \alpha 2 (1-\cos{kt})^\frac \alpha2$\!. In the general case, such modules were considered, in particular, in  \cite{Vasil'ev_2001}, \cite{Kozko_Rozhdestvenskii_2004},
\cite{Vakarchuk_2016}, \cite{Babenko_Konareva_2019}, etc.

%%%%%%%%%%%%%%%%%%%%%%%%%%%%%%%%%%%%%%%%%%%%%%%%%%%%%%%%%%%%%%%%%%%%%%%%%%%%%%%%%%%%%%%%%%%%%%%%%%%%%%%%%%%%

%%%%%%%%%%%%%%%%%%%%%%%%%%%%%%%%%%%%%%%%%%%%%%%%%%%%%%%%%%%%%%%%%%%%%%%%%%%%%%%%%%%%%%%%%%%%%%%%%%%%%%%%%%%%%%%%%%%%%%%%%%%%%%%%%

Further, let  ${\mathcal M}(\tau )$, $\tau>0$, be the set of all functions $\mu  $, bounded
non-decreasing and non-constant on the segment   $[0, \tau]$. By $\Omega _\varphi(f, \tau, \mu  , u)_{_{\scriptstyle  {\bf M}}}^* $, $u>0$,   denote the average value of the generalized modulus of smoothness
$\omega _\varphi(f, t)_{_{\scriptstyle  {\bf M}}}^*$ of the function  $f$ with the weight $\mu   \in {\mathcal M}(\tau )$, that is,
 \begin{equation}\label{Mean_Value_Gen_Modulus}
  \Omega _\varphi(f, \tau, \mu  , u)_{_{\scriptstyle  {\bf M}}}^*  := \frac
   {1}{\mu    (\tau ) - \mu    (0)}\int %\limits
   _0^u\omega _\varphi(f, t)_{_{\scriptstyle  {\bf M}}}^*  {\mathrm d}\mu    \Big(\frac {\tau
   t}{u}\Big).
 \end{equation}
%%%%%%%%%%%%%%%%%%%%%%%%%%%%%%%%%%%%%%%%%%%%%%%%%%%%%%%%%%%%%%%%%%%%%%%%%%%%%%%%%%%%%%%%%%%%%%%%%%%%%%%%%%%%%%%%%%%%%%%%%%%%%%%%%%%
%%%%%%%%%%%%%%%%%%%%%%%%%%%%%%%%%%%%%%%%%%%%%%%%%%%%%%%%%%%%%%%%%%%%%%%%%%%%%%%%%%%%%%%%%%%%%%%%%%%%%%%%%%%%%%%%%%%%%%%%%%%%%%%%%%%

Note that for any $f\in {\mathcal S}_{\bf M},$ $\tau >0,$ $\mu
\in {\mathcal M}(\tau )$ and  $u>0$ the functionals $\Omega _\varphi(f, \tau, \mu  , u)_{_{\scriptstyle  {\bf M}}}^* $
do not exceed the value $\omega _\varphi(f,u)_{_{\scriptstyle  {\bf M}}}^*,$ and therefore in a number of questions
they can be more effective for characterizing the structural and approximative properties of the function $f$.

%%%%%%%%%%%%%%%%%%%%%%%%%%%%%%%%%%%%%%%%%%%%%%%%%%%%%%%%%%%%%%%%%%%%%%%%%%%%%%%%%%%%%%%%%%%%%%%%%%%%%%%%%%%%

%%%%%%%%%%%%%%%%%%%%%%%%%%%%%%%%%%%%%%%%%%%%%%%%%%%%%%%%%%%%%%%%%%%%%%%%%%%%%%%%%%%%%%%%%%%%%%%%%%%%%%%%%%%%

\subsection{Definition of $\psi$-integrals, $\psi$-derivatives   and functional classes}

Let  $\psi=\{\psi(k)\}_{k\in {\mathbb Z}}$ be an arbitrary sequence of complex numbers.
If for a given function $f\in L$ with the Fourier series
 $\sum%\limits
 _{k\in {\mathbb Z}}\widehat {f}(k){\mathrm e}^{{\mathrm i}k
 x}
 $
 the series
 $
 \sum%\limits
 _{k\in {\mathbb Z}}\psi(k)\widehat {f}(k){\mathrm e}^{{\mathrm i}k
 x}
 $
 is the Fourier series of  a certain function  $F\in L$, then   $F$ is called  (see, for example, \cite[Ch.~11]{Stepanets_M2005})
 $\psi$-integral of the function  $f$ and is denoted as $F={\cal J}^{\psi }(f, \cdot)$. In turn,
 the function $f$ is called the $\psi$-derivative of the function $F$ and is denoted as  $f=F^{\psi}$.
 The Fourier coefficients of functions $f$ and $f^{\psi }$ are related by the equalities
\begin{equation} \label{Fourier_Coeff_der}
 \widehat f(k)=\psi (k)\widehat f^{\psi }(k), \quad k \in {\mathbb Z}.
  \end{equation}
The set of  $\psi$-integrals of
functions  $f$ of $L$ is denoted as  $L^{\psi }$. If
 ${\mathfrak {N}}\subset L$, then  $L^{\psi }{\mathfrak {N}}$ denotes the set of $\psi$-integrals of
 functions $f\in {\mathfrak {N}}$. In particular,    $L^{\psi }{\mathcal S}_{\bf M}$ is the set
 of $\psi$-integrals of    functions $f\in {\mathcal S}_{\bf M}$.

%In the case when  $\psi (k) = ({\mathrm i}k)^{-r}$, $r=0, 1,\ldots$, we set  $L^{\psi }=:L^{r}$ and $L^{\psi }{\mathfrak {N}}=:L^{r}{\mathfrak {N}}$.

%%%%%%%%%%%%%%%%%%%%%%%%%%%%%%%%%%%%%%%%%%%%%%%%%%%%%%%%%%%%%%%%%%%%%%%%%%%%%%%%%%%%%
For arbitrary fixed $\varphi\in \Phi$, $\tau>0$ and $\mu  \in  M(\tau)$, define the following functional classes:
  \begin{equation} \label{L^psi(varphi,n)}
  L^{\psi }(\varphi,\tau ,\mu   , n)_{_{\scriptstyle  {\bf M}}}^*:=
  \Big\{f\in L^{\psi }{\mathcal S}_{\bf M}:\quad
  \Omega_\varphi\Big(f^{\psi }, \tau, \mu   ,\frac{\tau }n\Big)_{_{\scriptstyle  {\bf M}}}^* \le 1, \quad n \in
  {\mathbb{N}}\Big\},
  \end{equation}%\eqno (\ref{L^psi(varphi,n)})$$
   \begin{equation} \label{L^psi(varphi,Phi)}
   L^{\psi }(\varphi, \tau, \mu   , \Omega )_{_{\scriptstyle  {\bf M}}}^*  :=
    \Big\{f\in L^{\psi }{\mathcal S}_{\bf M}:\
    \Omega _\varphi(f^{\psi }, \tau , \mu   ,  u)_{_{\scriptstyle  {\bf M}}}^* \le \Omega  (u),\ 0\le u\le \tau \Big\},
  \end{equation}%\eqno (6.86)$$
where $\Omega  (u)$ is a fixed continuous monotonically increasing function of the variable   $u\ge 0$ such that $\Omega  (0)=0$.
%%%%%%%%%%%%%%%%%%%%%%%%%%%%%%%%%%%%%%%%%%%%%%%%%%%%%%%%%%%%%%%%%%%%%%%%%%%%%%%%%%%%%%%%%%%%%%%%%%%%%%%%%%%%%%%%%%%%%%%%%%%
%In particular,  we also set  $L^{\psi }(\alpha,\tau ,\mu   , n)_{_{\scriptstyle  {\bf M}}}^*:=L^{\psi }(\varphi,\tau ,\mu   , n)_{_{\scriptstyle  {\bf M}}}^*$ and  $L^{\psi }(\alpha,\tau ,\mu  ,  \Omega )_{_{\scriptstyle  {\bf M}}}^*:=
%L^{\psi }(\varphi, \tau, \mu   , \Omega )_{_{\scriptstyle  {\bf M}}}^*$ for  $\varphi(t)=\varphi_\alpha(t)=
%2^\frac \alpha 2 (1-\cos{kh})^\frac \alpha2$.
%%%%%%%%%%%%%%%%%%%%%%%%%%%%%%%%%%%%%%%%%%%%%%%%%%%%%%%%%%%%%%%%%%%%%%%%%%%%%%%%%%%%%%%%%%%%%%%%%%%%%%%%%%%%%%%%%%%%%%%%%%%

Note that  in the Lebesgue spaces $L_2$,  for  $\psi (k)=k^{-r}$, $r\in \mathbb{N}$,
$\varphi(t)=2^\frac \alpha 2 (1-\cos{kt})^\frac \alpha2$, and the weight function
$\mu  (t)=t$, Taikov     \cite{Taikov_1976, Taikov_1979} (see also \cite[Ch. 4]{Pinkus_1985}) first considered the
functional classes  similar to (\ref{L^psi(varphi,n)}) and (\ref{L^psi(varphi,Phi)}).  He found
the exact values of the widths of such classes  in the spaces  $L_2$   in the case when the majorants $\Omega$
of the averaged values of  the moduli of smoothness satisfied some constraints. The problem of finding the exact values
of the widths
in different  spaces of functional classes of this kind %generated by some specific weighting functions $\mu  $,
was also studied in   \cite{Yussef_1988, Shalaev_1991, Esmaganbetov_1999, Serdyuk_2003, Vakarchuk_2004, Vakarchuk_2016, Abdullayev_Serdyuk_Shidlich_2021}, etc.

%%%%%%%%%%%%%%%%%%%%%%%%%%%%%%%%%%%%%%%%%%%%%%%%%%%%%%%%%%%%%%%%%%%%%%%%%%%%%%%%%%%%%%%%%%%%%%%%%%%%%%%%%%%%%

\subsection{Best approximations and widths of functional classes}

%%%%%%%%%%%%%%%%%%%%%%%%%%%%%%%%%%%%%%%%%%%%%%%%%%%%%%%%%%%%%%%%%%%%%%%%%%%%%%%%%%%%%%%%%%%%%%%%%%%%%%%%%%%%%

Let   ${\mathscr T}_{2n+1}$, $n=0,1,\ldots$, be the set of all trigonometric polynomials
${T}_{n}(x) = \sum_{|k|\le n}  c_{k}\mathrm{e}^{\mathrm{i}kx}$ of the order  $n$, where $c_{ k}$
are arbitrary complex numbers.

For any function $f\in {\mathcal S}_{\bf M}$ denote by $E_n (f)_{_{\scriptstyle  {\bf M}}}^*$ its best approximation by the trigonometric polynomials   ${T}_{n-1}\in {\mathscr T}_{2n-1}$ in the space  ${\mathcal S}_{\bf M}$ with respect to the norm $\|\cdot\|_{_{\scriptstyle  {\bf M}}}^*$, i.e.,
 \begin{equation}\label{Best_Approx_Deff}
    E_n (f)_{_{\scriptstyle  {\bf M}}}^* :=
    \inf\limits_{{T}_{n-1}\in {\mathscr T}_{2n-1} }\|f-{T}_{n-1}\|_{_{\scriptstyle  {\bf M}}}^*
    .
 \end{equation}
%%%%%%%%%%%%%%%%%%%%%%%%%%%%%%%%%%%%%%%%%%%%%%%%%%%%%%%%%%%%%%%%%%%%%%%%%%%%%%%%
From  relation (\ref{def-Orlicz-norm}),  it follows \cite[Lemma 2]{Abdullayev_Chaichenko_Shidlich_2021} that for any   $f \in {\mathcal S}_{\bf M}$ and all $n=0,1,\ldots$,
           \begin{equation} \label{Best_app_all}
           E_n (f)_{_{\scriptstyle  {\bf M}}}^* =\|f-{S}_{n-1}({f})\|_{_{\scriptstyle  {\bf M}}}^*
             = \sup \Big\{
           \sum _{|k |\ge  n}  \lambda_k|\widehat{f}(k) |: \quad  \lambda\in \Lambda \Big\},
           \end{equation}
where $S_{n-1}(f)=S_{n-1}(f,\cdot)= \sum _{|k|\le n-1}\widehat{f}(k) {\mathrm{e}^{\mathrm{i}k\cdot}}$ is the partial Fourier sum
of the order  $n-1$ of the function $f$.

%%%%%%%%%%%%%%%%%%%%%%%%%%%%%%%%%%%%%%%%%%%%%%%%%%%%%%%%%%%%%%%%%%%%%%%%%%%%%%%%

Further, let  $K$ be a convex centrally symmetric subset of ${\mathcal S}_{\bf M}$ and
 let $B_{_{\scriptstyle  {\bf M}}}^*$ be a unit ball of the space  ${\mathcal S}_{\bf M}$ with respect to the norm $\|\cdot\|_{_{\scriptstyle  {\bf M}}}^*$. Let also  $F_N$ be an arbitrary $N$-dimensional subspace of the space ${\mathcal S}_{\bf M}$, $N\in {\mathbb N}$, and
$\mathscr{L}({\mathcal S}_{\bf M}, F_N)$ be a set of linear operators from ${\mathcal S}_{\bf M}$ to $F_N$.
 By    $\mathscr {P}({\mathcal S}_{\bf M}, F_N)$ denote the subset of projection operators of the set ${\mathscr{L}}({\mathcal S}_{\bf M}, F_N)$,  that is, the set of the operators  $A$ of linear projection onto the set
  $F_N$ such that $Af = f$ when $f\in F_N$.  The quantities
 \[
 b_N(K, {\mathcal S}_{\bf M})=\sup\limits _{F_{N+1}}\sup\{\varepsilon>0: \varepsilon B_{_{\scriptstyle  {\bf M}}}^*\cap F_{N+1}
 \subset K\},
 \]
 \[
 d_N(K, {\mathcal S}_{\bf M})=\inf\limits _{F_N}\sup \limits _{f\in K}
 \inf \limits _{u\in F_N}\|f - u \|_{_{\scriptstyle  {\bf M}}}^* ,
 \]
 \[
 \lambda _N(K,{\mathcal S}_{\bf M})=
 \inf \limits _{F_N}\inf \limits_{A\in {\mathscr {L}}({\mathcal S}_{\bf M}, F_N)}\sup \limits _{f\in K}
 \|f - Af\|_{_{\scriptstyle  {\bf M}}}^* ,
 \]
 \[
 \pi _N(K, {\mathcal S}_{\bf M})=\inf \limits_{F_N}\inf \limits _{A\in {\mathscr {P}}({\mathcal S}_{\bf M},F_N)}
 \sup \limits _{f\in K}\|f - Af\|_{_{\scriptstyle  {\bf M}}}^* ,
 \]
 are called Bernstein, Kolmogorov, linear, and projection $N$-widths of the set $K$ in the space ${\mathcal S}_{\bf M}$, respectively.

%%%%%%%%%%%%%%%%%%%%%%%%%%%%%%%%%%%%%%%%%%%%%%%%%%%%%%%%%%%%%%%%%%%%%%%%%%%%%%%%%%%%%%%%%%%%%%%%%%%%%%%%%%%%%

\section{Main results}

\subsection{Jackson-type inequalities}

In this subsection,  Jackson-type inequalities are obtained  in terms of the   best approximations and the  averaged values of generalized moduli of smoothness in the spaces ${\mathcal S}_{\bf M}$. To state these results,  denote by $\Psi$ the set of arbitrary
sequences $\psi=\{\psi (k)\}_{k\in {\mathbb Z}}$  of complex numbers such that  $|\psi (k)|=|\psi (-k)| \ge |\psi (k+1)|$ for $k\in {\mathbb N}$. Here and below, we also assume that the sequence ${\bf {M^*}}=\{{M^*_k}(v)\}_{k\in {\mathbb Z}}$  satisfies the condition \begin{equation}\label{M_condition}
M_k^*(v)>M_k^*(1)=1, \quad v>1,\ k\in {\mathbb Z}.
 \end{equation}%\eqno (6.87)$$,

 \begin{theorem}
       \label{Th.3.1}
       Assume that $f\in L^{\psi }{\mathcal S}_{\bf M}$, condition $(\ref{M_condition})$ holds,    $\varphi\in \Phi$, $\tau>0$,  $\mu  \in  {\mathcal M}(\tau )$
       and $\psi\in \Psi$.         Then for any   $n\in {{\mathbb N}}$ the following inequality is true:
 \begin{equation}\label{Jackson_Type_Ineq}
    E_n(f)_{_{\scriptstyle  {\bf M}}}^* \le
     \frac {\mu   (\tau ) - \mu   (0)}
    {I_{n,\varphi}(\tau ,\mu   )}  |\psi  (n)|\,
    \Omega _\varphi\Big(f^{\psi}, \tau,\mu   , \frac{\tau }n\Big)_{_{\scriptstyle  {\bf M}}}^*,
 \end{equation}%\eqno (6.87)$$,
where
 \begin{equation}\label{I_n,varphi,p}
      I_{n,\varphi}(\tau ,\mu   ):=
      \mathop{\inf\limits _{k\ge n}}\limits_{k \in {\mathbb N}} \int\limits _0^{\tau }\varphi \Big(\frac {k t}n\Big)
      {\mathrm d} \mu     (t).
 \end{equation}%\eqno (6.88)$$
If, in addition, the function $\varphi$ is non-decreasing on the interval
$[0,\tau]$  and the condition
 \begin{equation}\label{I_n,varphi,p_Equiv_n}
      I_{n,\varphi }(\tau ,\mu   )=\int\limits _0^{\tau }\varphi (t) {\mathrm d} \mu     (t),
 \end{equation}%(6.89)
holds, then  inequality $(\ref{Jackson_Type_Ineq})$ can not be improved and therefore,
 \begin{equation}
       \label{Jackson_Type_Exact}
       \mathop {\sup\limits _{f\in L^{\psi}{\mathcal S}_{\bf M}}}\limits _{f\not ={\rm const }}
       \frac {E_n(f)_{_{\scriptstyle  {\bf M}}}^* }
       {\Omega_\varphi (f^{\psi }, \tau, \mu  , \frac{\tau }{n} )_{_{\scriptstyle  {\bf M}}}^* }=
        \frac {\mu  (\tau ) - \mu   (0)}{\int_0^{\tau }\varphi (t) {\mathrm d} \mu     (t)} |\psi  (n)|.
 \end{equation}%(6.90)$$

\end{theorem}

\begin{proof} Let $f\in L^{\psi }{\mathcal S}_{\bf M}$. By virtue of  (\ref{Fourier_Coeff_der}) and (\ref{Best_app_all}), we have
 \[
 E_n(f)_{_{\scriptstyle  {\bf M}}}^* =\sup \Big\{
           \sum _{|k |\ge  n}  \lambda_k|\widehat{f}(k) |: \  \lambda\in \Lambda \Big\} \le
           \sup \Big\{
           \sum _{|k |\ge  n}  \lambda_k \Big|\frac {\psi (n)
 }{\psi (k)}\Big||\widehat{f}(k) |: \   \lambda\in \Lambda \Big\}
 \]
 \begin{equation}
       \label{A6.91}
       =|\psi  (n)|  \sup \Big\{
           \sum _{|k |\ge  n}  \lambda_k \Big|\frac {\widehat f(k)}{\psi (k)}\Big|: \quad  \lambda\in \Lambda \Big\} =
           |\psi  (n)|  E_n (f^{\psi })_{_{\scriptstyle  {\bf M}}}^* .
   \end{equation}%         \eqno(6.91)\]
As shown in  \cite[Proof of Theorem 1]{Abdullayev_Chaichenko_Shidlich_2021}, for any   $g \in {\mathcal S}_{\bf M}$,   $\tau >0$, $\varphi\in \Phi$,  $\mu   \in {\mathcal M}(\tau )$ and   $n\in {\mathbb N}$
 \begin{equation}\label{Jackson_type_OLD}
       E_n (g)_{_{\scriptstyle  {\bf M}}}^* \le
       \frac 1{I_{n,\varphi}(\tau ,\mu   )}
       \int \limits _0^{\tau }
       \omega_\varphi \Big(g, \frac t{n}\Big)_{_{\scriptstyle  {\bf M}}}^* {\mathrm d} \mu    (t).
   \end{equation}%(6.92)$$
Let us set  $g =f^{\psi }$ in (\ref{Jackson_type_OLD}). Then
   \begin{equation}\label{A6.93}
        E_n (f^{\psi})_{_{\scriptstyle  {\bf M}}}^* \le
        \frac {\mu   (\tau ) - \mu   (0)}{I_{n,\varphi}(\tau ,\mu   )}
        \frac {\int  _0^{\tau }\omega _\varphi (f^{\psi }, \frac t{n})_{_{\scriptstyle  {\bf M}}}^* {\mathrm d} \mu    (t)}{\mu   (\tau ) - \mu   (0)} \le %=
        \frac {\mu   (\tau ) - \mu   (0)}{I_{n,\varphi}(\tau ,\mu   )}
        \Omega _\varphi \Big(f^{\psi }, \tau, \mu   , \frac{\tau }n\Big)_{_{\scriptstyle  {\bf M}}}^* .
   \end{equation}%\eqno (6.93)$$
Combining  inequalities (\ref{A6.91}) and (\ref{A6.93}), we obtain   (\ref{Jackson_Type_Ineq}).

Now assume that the function $\varphi$ is non-decreasing on  the interval
$[0,\tau]$ and  %the
condition (\ref{I_n,varphi,p_Equiv_n}) holds.
Then by virtue of (\ref{Jackson_Type_Ineq}), we have
 \begin{equation}\label{A6.94}
      \mathop {\sup\limits _{f\in L^{\psi}{\mathcal S}_{\bf M}}}\limits _{f\not ={\rm const }}
       \frac {E_n(f)_{_{\scriptstyle  {\bf M}}}^* }
       {\Omega_\varphi (f^{\psi }, \tau, \mu  , \frac{\tau }n )_{_{\scriptstyle  {\bf M}}}^* }\le
             \frac {\mu  (\tau ) - \mu   (0)}{\int_0^{\tau }\varphi(u) {\mathrm d} \mu     (u)}|\psi  (n)|
       .
   \end{equation}%\eqno (6.94)$$
 To prove the unimprovability of  inequality (\ref{A6.94}),   consider  the function
 \[
  f_n(x)=\gamma  +\varepsilon_{-n}\delta   {\mathrm e}^{-{\mathrm i}nx} +
  \varepsilon_{n} \delta    {\mathrm e}^{{\mathrm i}nx},
  \]
where  $\gamma$ and  $\delta $ are arbitrary complex numbers, and   $\varepsilon_k$, $k\in \{-n,n\}$,
are integers such that  $|\varepsilon_{n}|+|\varepsilon_{-n}|=1$%( $\varepsilon_{n}=\pm 1$ and $\varepsilon_{-n}=0$  or $\varepsilon_{n}=0$ and $\varepsilon_{-n}=\pm 1$)
. Taking into account relations (\ref{modulus_generalize difference_Fourier_Coeff}), (\ref{Fourier_Coeff_der}) and (\ref{M_condition}), we have
 \begin{equation}\label{A6.95}
 \|\Delta_h^\varphi f_n^\psi\|_{_{\scriptstyle  {\bf M}}}^*=|\delta |\frac {\varphi (n h)}{|\psi(n)| }\sup \Big\{
           |\varepsilon_{-n}|\lambda_{-n}+
           |\varepsilon_{n}|\lambda_{n}: \quad  \lambda\in \Lambda \Big\}
 =  |\delta |\frac {\varphi (nh)}{|\psi(n)| }.
 \end{equation}%\eqno (6.95)$$
Since the  function $\varphi(nh)$ is non-decreasing on the interval  $[0, \frac {\tau}n]$,
then for $0\le t\le \tau$,
 \begin{equation}\label{A6.96}
\omega _\varphi (f_n^\psi, t) =    |\delta  |
 \frac {\varphi
 (nt)}{|\psi  (n)|}.
 \end{equation}% \eqno (6.96)$$
Taking into account (\ref{Mean_Value_Gen_Modulus}),  (\ref{A6.96}) and the equality
 $
    E_n
  (f_n)_{_{\scriptstyle  {\bf M}}}^*  =|\delta  |
  ,
 $
we see that
  \[
  \mathop {\sup\limits _{f\in L^{\psi }{\mathcal S}_{\bf M}}}\limits _{f\not ={\rm const}}
  \frac {E_n  (f)_{_{\scriptstyle  {\bf M}}}^* }
  {\Omega _\varphi  (f^{\psi }, \tau , \mu   , \frac {\tau }n)_{_{\scriptstyle  {\bf M}}}^* }
  \ge
  \frac {E_n  (f_n)_{_{\scriptstyle  {\bf M}}}^* }
  {\Omega _\varphi  (f_n^{\psi }, \tau, \mu  , \frac {\tau }n)_{_{\scriptstyle  {\bf M}}}^* }
  \]
  \begin{equation}\label{A6.97}
  = \frac {|\delta  | (\mu   (\tau )-\mu   (0))  |\psi  (n)|}
   {\int _0^{\tau /n}|\delta  | \varphi (nt){\mathrm d} \mu    (nt) }=
         \frac {\mu  (\tau ) - \mu   (0)}{\int_0^{\tau }\varphi(u) {\mathrm d} \mu     (u)}  |\psi  (n)|.
 \end{equation}% \eqno (6.97)$$
Relations (\ref{A6.94}) and (\ref{A6.97}) yield (\ref{Jackson_Type_Exact}).  %$\hfill\Box$

\end{proof}

%%%%%%%%%%%%%%%%%%%%%%%%%%%%%%%%%%%%%%%%%%%%%%%%%%%%%%%%%%%%%%%%%%%%%%%%%%%%%%%%%%%%%%%%%%%%%%%%%%%%%%
%%%%%%%%%%%%%%%%%%%%%%%%%%%%%%%%%%%%%%%%%%%%%%%%%%%%%%%%%%%%%%%%%%%%%%%%%%%%%%%%%%%%%%%%%%%%%%%%%%%%%%

\subsection{Widths of the classes $L^{\psi }(\varphi, \mu, \tau, n)_{_{\scriptstyle  {\bf M}}}^*  $}

In this subsection,  the values of Kolmogorov, Bernstein, linear, and projection widths are found for the
 classes $L^{\psi}(\varphi, \mu  , \tau, n)_{_{\scriptstyle  {\bf M}}}^*  $ in the case when the sequences
$\psi (k)$ satisfy some natural restrictions.

\begin{theorem} \label{Th.3.2}
       Assume that    $\psi \in \Psi$, $\tau>0$,  condition $(\ref{M_condition})$ holds, the function
       $\varphi\in \Phi$ is non-decreasing on the interval $[0,\tau]$ and  $\mu   \in {\mathcal M}(\tau )$.
       Then for any  $n\in {\mathbb N}$ and  $N\in \{2n-1,  2n\}$  the following inequalities are true:
\begin{equation}\label{A6.105}
    \frac{\mu   (\tau ) - \mu   (0)}
   {\int_0^{\tau }\varphi (t) {\mathrm d} \mu     (t)} |\psi (n)|
   \le P_{N} \Big(L^{\psi }(\varphi, \tau, \mu  , n)_{_{\scriptstyle  {\bf M}}}^*  , {\mathcal S}_{_{\scriptstyle  {\bf M}}} \Big)
    \le  \frac {\mu   (\tau ) - \mu  (0)}
   {I_{n,\varphi}(\tau ,\mu   )}    |\psi (n)|,
  \end{equation}%\eqno (6.105)$$
where the quantity $I_{n,\varphi}(\tau ,\mu   )$ is defined by $(\ref{I_n,varphi,p})$, and $P_N$
is any of the widths $b_N$, $d_N$, $\lambda _N$ or $ \pi _N$. If, in addition, condition $(\ref{I_n,varphi,p_Equiv_n})$ holds, then
    \begin{equation}\label{A6.106} %\[
          P_{N} \Big(L^{\psi }(\varphi, \tau, \mu  , n)_{_{\scriptstyle  {\bf M}}}^*  , {\mathcal S}_{_{\scriptstyle  {\bf M}}}\Big )=
           \frac{\mu   (\tau ) - \mu   (0)}
          {\int_0^{\tau }\varphi(t) {\mathrm d} \mu     (t)} |\psi (n)|.
      \end{equation}%\eqno (6.106)$$
\end{theorem}

\begin{proof}  The proof of Theorems \ref{Th.3.2} and  \ref{Th.3.3}  basically repeats the proof  of corresponding
theorems in the spaces   ${\mathcal S}^p$ (see \cite{Serdyuk_2003, Abdullayev_Serdyuk_Shidlich_2021}) and   is adapted  in accordance with the properties of the spaces ${\mathcal S}_{\bf M}$. Based on Theorem \ref{Th.3.1}, taking into account the definition of the set $\Psi$, for an arbitrary function
$f\in L^{\psi }(\varphi, \tau , \mu  , n)_{_{\scriptstyle  {\bf M}}}^*  $, we have
     \begin{equation}\label{A6.107}
            E_n(f)_{_{\scriptstyle  {\bf M}}}^*  \le
                 \frac {\mu   (\tau ) - \mu  (0)}{I_{n,\varphi}(\tau ,\mu   )}
                 \Omega _\varphi\Big(f^{\psi}, \tau , \mu  , \frac {\tau}n\Big)_{_{\scriptstyle  {\bf M}}}^*  |\psi (n)|
      \le   \frac {\mu   (\tau ) - \mu  (0)}{I_{n,\varphi}(\tau ,\mu   )}
                |\psi (n)|.
      \end{equation}%\eqno(6.107)$$
Then, taking into account the definition of the projection width $\pi _{N}$, and  relations
(\ref{Best_app_all}) and  (\ref{A6.107}), we conclude that
     \begin{equation}\label{A6.109} %\[
           \pi _{2n-1}\Big(L^{\psi }(\varphi, \tau, \mu  , n)_{_{\scriptstyle  {\bf M}}}^*  , {\mathcal S}_{_{\scriptstyle  {\bf M}}}\Big)
           \le   \sup\limits _{f\in L^{\psi }(\varphi, \tau,\, \mu  ,\, n)_{_{\scriptstyle  {\bf M}}}^*  }
          E_n(f)_{_{\scriptstyle  {\bf M}}}^*
 \le
           \frac {\mu   (\tau ) - \mu  (0)}{I_{n,\varphi}(\tau ,\mu   )} |\psi (n)|.
      \end{equation}%\eqno(6.109)$$
Since the widths  $b_N$, $d_N$, $\lambda _N$ and  $\pi _N$ do not increase with increasing $N$  and
     \begin{equation}\label{A6.110}
         b_N(K, X)\le   d_N(K, X) \le \lambda_N(K, X)\le \pi _N(K, X)
      \end{equation}%           \eqno (6.110)$$
(see, for example, \cite[Ch.~4]{Tikhomirov_M1976}), then by virtue of (\ref{A6.109}), we get  the upper estimate in
(\ref{A6.105}).

To obtain the necessary lower estimate, it suffices to show that
     \begin{equation}\label{A6.111} % $$
              b_{2n}\Big(L^{\psi }(\varphi, \mu  , \tau, n)_{_{\scriptstyle  {\bf M}}}^*  , {\mathcal S}_{_{\scriptstyle  {\bf M}}}\Big)\ge
               \frac{\mu   (\tau ) - \mu   (0)}   {\int_0^{\tau }\varphi(u) {\mathrm d} \mu     (u)} |\psi (n)|=:R_n.
      \end{equation}%                         \eqno (6.111)$$
%%%%%%%%%%%%%%%%%%%%%%%%%%%%%%%%%%%%%%%%%%%%%%%%%%%%%%%%%%%%%%%%%%%%%%%%%%%%%%%%%%%%%%%%%%%%%%%%%%%%%%%%%%%%%%%%%%
In the $(2n+1)$-dimensional space ${\mathscr T}_{2n+1}$ of trigonometric polynomials of order $n$, consider
the ball $B_{2n+1}$, whose radius is equal to the number $R_n$ defined in  (\ref{A6.111}), that is,
  $$
  B_{2n+1}=\Big\{t_n\in {\mathscr T}_{2n+1} % (x)=\sum _{|k|\le n}c_k{\mathrm e}^{{\mathrm i}kx},  c_k\in \mathbb {C}
  : \|t_n\|_{_{\scriptstyle  {\bf M}}}^*  \le R_n\Big\}.
  $$
and prove the embedding  $B_{2n+1}\subset L^{\psi }(\varphi, \tau, \mu  , n)_{_{\scriptstyle  {\bf M}}}^*  $.

%%%%%%%%%%%%%%%%%%%%%%%%%%%%%%%%%%%%%%%%%%%%%%%%%%%%%%%%%%%%%%%%%%%%%%%%%%%%%%%%%%%%%%%%%%%%%%%%%%%%%%%%%%%%%%%%%%%%%%%%%%%%%%

For an arbitrary polynomial $T_n\in B_{2n+1}$, due to (\ref{general_modulus}) and the parity of the function
$\varphi$, we have
 $$
 \omega_\varphi(T_n^{\psi }, t)_{_{\scriptstyle  {\bf M}}}^*
 =\sup\limits_{0\le v\le t}\sup\Big\{
 \sum _{|k|\le n} \lambda_k\varphi(kv) |\widehat T_n^{\psi }(k)|:\lambda\in \Lambda\Big\} .
 $$
Then, taking into account   relation (\ref{Fourier_Coeff_der}) and the nondecreasing of the function $\varphi$  on $[0,a]$,
for $\tau \in (0,a]$  we get
  \[
  (\mu   (\tau ) -\mu   (0))  \Omega _\varphi \Big (T_n^{\psi }, \tau,\mu  , \frac {\tau }n\Big )_{_{\scriptstyle  {\bf M}}}^*
   =\int\limits _0^{\tau }\omega _\varphi  \Big(T_n^{\psi }, \frac t{n}\Big)_{_{\scriptstyle  {\bf M}}}^*  {\mathrm d} \mu    (t)
   \]
     \[
    = \int\limits _0^{\tau } \sup\limits_{0\le v\le \frac t{n}}\sup\Big\{
 \sum _{|k|\le n} \lambda_k\varphi(kv) |\widehat T_n^{\psi }(k)|:\lambda\in \Lambda\Big\}  {\mathrm d} \mu    (t)
 \]
 \[
 = \int\limits _0^{\tau } \sup\limits_{0\le v\le \frac t{n}}\sup\Big\{
 \sum _{|k|\le n} \lambda_k\varphi(kv) \Big| \frac{\widehat T_n(k)}{\psi(k)}\Big|:\lambda\in \Lambda\Big\}  {\mathrm d} \mu    (t)
 \]
 \[
 \le \frac 1{|\psi(n)|}\int\limits _0^{\tau } \varphi(t) \sup\Big\{
 \sum _{|k|\le n} \lambda_k  |  \widehat T_n(k) |:\lambda\in \Lambda\Big\}  {\mathrm d} \mu    (t)
 = \frac {\|T_n\|_{_{\scriptstyle  {\bf M}}}^*    }
    {|\psi(n)| }
    \int\limits _0^{\tau } \varphi(t) {\mathrm d} \mu    (t).
   \]
Therefore, given the inclusion $T_n\in B_{2n+1}$ it follows that
$\Omega _\varphi   (T_n^{\psi }, \tau,\mu  , \frac {\tau }n  )_{_{\scriptstyle  {\bf M}}}^* \!\le \!1$. Thus,
$T_n\in L^{\psi }(\varphi, \tau, \mu  , n)_{_{\scriptstyle  {\bf M}}}^* $, and the embedding
$B_{2n+1}\subset L^{\psi }(\varphi,\mu   , \tau, n)_{_{\scriptstyle  {\bf M}}}^* $ is true.  By
the definition of  Bernstein width, the inequality  (\ref{A6.111}) holds. Thus, relation  (\ref{A6.105}) is proved.
It is easy to see that, under the condition (\ref{I_n,varphi,p_Equiv_n}), the upper and lower bounds for
the quantities $P_N (L^{\psi}(\varphi, \tau, \mu  , n)_{_{\scriptstyle  {\bf M}}}^*  ,
{\mathcal S}_{_{\scriptstyle  {\bf M}}} )$ coincide and, therefore, equalities (\ref{A6.106}) hold.

\end{proof}

%%%%%%%%%%%%%%%%%%%%%%%%%%%%%%%%%%%%%%%%%%%%%%%%%%%%%%%%%%%%%%%%%%%%%%%%%%%%%%%%%%%%%%%%%%%%%%%%%%%%%%%%%%%%%%%%%%%%%%%%%%%

\subsection{Widths of the classes   $L^{\psi }(\varphi, \mu  , \tau, \Omega )_{_{\scriptstyle  {\bf M}}}^*  $}

Let us find  the widths of the classes $L^{\psi }(\varphi, \mu  , \tau,
\Omega )_{_{\scriptstyle  {\bf M}}}^*  $   that are defined by a majorant $\Omega$ of the averaged values of  generalized moduli of smoothness.

\begin{theorem}
       \label{Th.3.3}
       Let   $\psi \in \Psi$,  condition $(\ref{M_condition})$ holds, the function
       $\varphi\in \Phi$ be non-decreasing on a certain interval $[0,a]$, $a>0$, and
       $\varphi(a)=\sup\{\varphi(t):\, t\in {\mathbb R}\}$.
       Let also  $\tau\in (0,a]$,  the function  $\mu   \in {\mathcal M}(\tau )$
       and  $\Omega  (u)$ be a fixed continuous monotonically
       increasing function of the variable   $u \ge  0$  such that $\Omega  (0)=0$ and for all $\xi>0$ and $0<u\le a $,  the condition
     \begin{equation}\label{A6.113} %  $$
     \Omega  \Big(\frac u{\xi }\Big) \int\limits _0^{\xi \tau}\varphi_{*} (t) {\mathrm d} \mu
     \Big(\frac t{\xi }\Big)
     \le \Omega  (u)\ \int \limits _0^{\tau }\varphi(t)   {\mathrm d} \mu    (t)  ,
      \end{equation}%          \eqno (6.113)$$
      is satisfied, where
     \begin{equation}\label{A6.114} %        $$
      \varphi_{*}(t):=\left \{\begin{matrix} \varphi(t),\quad \hfill & 0\le t\le a, \\
      \varphi(a),\quad \hfill &  t\ge a.\end{matrix}\right.
      \end{equation}%  \eqno (6.114)$$
 Then for any  $n\in {\mathbb N}$ and  $N\in \{2n-1,  2n\}$  the following inequalities are true:
\[
             \frac{\mu   (\tau ) - \mu   (0)}{\int_0^{\tau }\varphi(t) {\mathrm d} \mu     (t)}
            |\psi (n)|\,\,\Omega  \Big(\frac {\tau }n\Big)\le
            P_N\Big(L^{\psi }(\varphi, \tau, \mu  , \Omega)_{_{\scriptstyle  {\bf M}}}^*  , {\mathcal S}_{_{\scriptstyle  {\bf M}}}\Big)
            \]
                    \begin{equation}\label{A6.115} %
                    \le
    \frac {\mu   (\tau ) - \mu  (0)} {I_{n,\varphi}(\tau ,\mu   )}   |\psi (n)|
            \,\,\Omega  \Big(\frac {\tau }n\Big),
      \end{equation}%   \eqno (6.115)$$
where the quantity $I_{n,\varphi}(\tau ,\mu   )$ is defined by $(\ref{I_n,varphi,p})$, and $P_N$
is any of the widths $b_N$, $d_N$, $\lambda _N$ or $ \pi _N$. If, in addition,  condition $(\ref{I_n,varphi,p_Equiv_n})$ holds, then
     \begin{equation}\label{A6.116} %
      P_{N}(L^{\psi } \Big(\varphi, \tau,  \mu  , \Omega  )_{_{\scriptstyle  {\bf M}}}^*  , {\mathcal S}_{_{\scriptstyle  {\bf M}}}\Big)=
      \frac{\mu   (\tau ) - \mu   (0)}{\int_0^{\tau }\varphi(t) {\mathrm d} \mu     (t)}
            |\psi (n)|\,\,\Omega  \Big(\frac {\tau }n\Big).
      \end{equation}%   \eqno (6.116)$$

\end{theorem}

\begin{proof}
Based on  inequality (\ref{Jackson_Type_Ineq}), for an arbitrary function
$f\in L^{\psi }(\varphi, \tau, \mu  , \Omega)_{_{\scriptstyle  {\bf M}}}^* $
   \begin{equation}\label{A6.117} %
            E_n(f)_{_{\scriptstyle  {\bf M}}}^*  \le
                 \frac {\mu   (\tau ) - \mu  (0)}{I_{n,\varphi}(\tau ,\mu   )}
                |\psi (n)| \Omega _\varphi\Big(\frac {\tau}n\Big),
      \end{equation}%
whence, taking into account the definition of the width $\pi _{N}$ and relation (\ref{Best_app_all}), we obtain
   \begin{equation}\label{A6.118} %\[
          \pi _{2n-1}\Big(L^{\psi }(\varphi, \mu  , \tau, \Omega  )_{_{\scriptstyle  {\bf M}}}^*  ,
          {\mathcal S}_{_{\scriptstyle  {\bf M}}}\Big)=\sup\limits _{f\in L^{\psi }(\varphi, \tau,\mu   , \Omega)_{_{\scriptstyle  {\bf M}}}^*  }\!\!
           E_n(f)_{_{\scriptstyle  {\bf M}}}^*
             \le  \frac {\mu   (\tau ) - \mu  (0)}{I_{n,\varphi}(\tau ,\mu   )}
                |\psi (n)| \Omega _\varphi\Big(\frac {\tau}n\Big).
      \end{equation}%  \eqno(6.118)$$
To obtain the necessary lower estimate, let us show that
    \begin{equation}\label{A6.119} % $$
    b_{2n}\Big(L^{\psi }(\varphi, \mu   , \tau, \Omega  )_{_{\scriptstyle  {\bf M}}}^*  ,{\mathcal S}_{_{\scriptstyle  {\bf M}}}\Big)\ge
     \frac {\mu   (\tau ) - \mu  (0)}{I_{n,\varphi}(\tau ,\mu   )}
    |\psi (n)| \Omega _\varphi\Big(\frac {\tau}n\Big)=: R_n^*.
      \end{equation}%  \eqno (6.119)$$
For this purpose, in the $(2n+1)$-dimensional space ${\mathscr T}_{2n+1}$ of trigonometric polynomials of order $n$, consider
ball $B_{2n+1}$, whose radius is equal to the number $R_n$ defined in (\ref{A6.119}), that is,
   \[ %$$
   B_{2n+1}^*=\Big\{T_n\in {\mathscr T}_{2n+1} % (x)=\sum _{|k|\le n}c_k{\mathrm e}^{{\mathrm i}kx},  c_k\in \mathbb {C}
  : \|T_n\|_{_{\scriptstyle  {\bf M}}}^*  \le R_n^*\Big\}
  \]
and prove the validity of the  embedding  $B_{2n+1}^*\subset L^{\psi }(\varphi, \mu   , \tau, \Omega  )_{_{\scriptstyle  {\bf M}}}^* $.

Assume that $T_n\in B_{2n+1}^*$. Taking into account the non-decrease of the function
$\varphi$ on $[0,a]$  and relations  (\ref{Fourier_Coeff_der}) and (\ref{A6.114}), we have
  \[
  (\mu   (\tau ) -\mu   (0))\cdot \Omega _\varphi  (T_n^{\psi }, \tau,\mu  , u )_{_{\scriptstyle  {\bf M}}}^*
   =\int\limits _0^{u}\omega _\varphi  (T_n^{\psi }, t)_{_{\scriptstyle  {\bf M}}}^*
   {\mathrm d} \mu    \Big(\frac {\tau t}{u}\Big)
   \]
 %   \[    = \int\limits _0^{\tau } \sup\limits_{0\le v\le t}\sup\Big\{  \sum _{|k|\le n} \lambda_k\varphi(kv) |\widehat T_n^{\psi }(k)|:\lambda\in \Lambda\Big\} {\mathrm d} \mu    \Big(\frac {\tau t}{u}\Big) \]
 \[
 = \int\limits _0^{\tau } \sup\limits_{0\le v\le t}\sup\Big\{
 \sum _{|k|\le n} \lambda_k\varphi(kv) \Big| \frac{\widehat T_n(k)}{\psi(k)}\Big|:\lambda\in \Lambda\Big\}  {\mathrm d} \mu    \Big(\frac {\tau t}{u}\Big)
 \]
 \[
 \le  \frac{1}{|\psi(n)|}\int\limits _0^{\tau } \varphi_*(n t)\sup\Big\{
 \sum _{|k|\le n} \lambda_k   | \widehat T_n(k)|:\lambda\in \Lambda\Big\}  {\mathrm d} \mu    \Big(\frac {\tau t}{u}\Big)
 \]
 \[
 =
  \frac {\|T_n\|_{_{\scriptstyle  {\bf M}}}^*   }
    {|\psi(n)|}
    \int\limits _0^{u} \varphi_*(n t) {\mathrm d} \mu    \Big(\frac {\tau t}{u}\Big)=
    \frac {\|T_n\|_{_{\scriptstyle  {\bf M}}}^*   }
    {|\psi(n)|}
    \int\limits _0^{nu} \varphi_*(t) {\mathrm d} \mu    \Big(\frac {\tau t}{nu}\Big).
   \]
 %%%%%%%%%%%%%%%%%%%%%%%%%%%%%%%%%%%%%%%%%%%%%%%%%%%%%%%%%%%%%%%%%%%%%%%%%%%%%%%%%%%%%%%%%%%%%%%%%%%%%%%%%
From the inclusion of  $T_n\in B_{2n+1}^*$ and  relation (\ref{A6.113}) with $\xi =\frac {nu}{\tau} $, it follows  that
   \[% \begin{equation}\label{A6.120} % $$
           \Omega _\varphi  (T_n^{\psi }, \tau,\mu  , u )_{_{\scriptstyle  {\bf M}}}^* \le
             \frac {\int _0^{nu} \varphi_*(t) {\mathrm d} \mu    (\frac {\tau t}{n u})}
           {\int_0^{\tau }\varphi(t) {\mathrm d} \mu     (t)}
           \Omega  \Big(\frac {\tau}n\Big)\le \Omega  (u).
    \] %\end{equation}%  \eqno (6.120)$$
Therefore, indeed  $B_{2n+1}^*\subset L^{\psi }(\varphi, \tau,\mu   , \Omega  )_{_{\scriptstyle  {\bf M}}}^*  $
and by the definition of Bernstein width, relation (\ref{A6.119}) is true. Combining relations
(\ref{A6.110}), (\ref{A6.117}) and (\ref{A6.119}), and taking into account
monotonic non-increase of each of the widths $b_N,$  $d_N,$ $\lambda _N$ and $\pi _N$ on $N$, we get
(\ref{A6.115}).  Under the additional condition (\ref{I_n,varphi,p_Equiv_n}), the upper and lower estimates of the quantities
$ P_N\Big(L^{\psi }(\varphi,  \tau, \mu  , \Omega  )_{_{\scriptstyle  {\bf M}}}^*  , {\mathcal S}_{_{\scriptstyle  {\bf M}}}\Big)$ coinside in relation (\ref{A6.115}) and hence, equalities (\ref{A6.116}) are true.

\end{proof}

\subsection{Some corollaries and remarks}

As mentioned above, the functional classes  similar to the classes  of the kind (\ref{L^psi(varphi,n)}) and (\ref{L^psi(varphi,Phi)}) were studied by many  authors. To compare our results with the results of these studies, let us give some notations.

Consider the case where all functions $M_k(t)=t^p\Big(p^{-1/p}q^{-1/q}\Big)^p$, $p>1$, $1/p+1/q=1$.
  In this case, all functions $M^*_k(v)=v^q$, the set $\Lambda({\bf {M^*}})$ is a set of all sequences of
  positive numbers $\lambda=\{\lambda_k\}_{k\in \mathbb{Z}}$ such that  $\|\lambda\|_{_{l_q({\mathbb Z})}}\le 1$.
  Then the spaces ${\mathcal S}_{\bf M}$ coincide with the  mentioned above spaces ${\mathcal S}^{p}$   of functions $f\in L$  with the finite norm
 \begin{equation}\label{norm_Sp}
    \|f\|_{_{\scriptstyle {p}}}:=
    \|f\|_{_{\scriptstyle {\mathcal S}^p}}=\|\{\widehat f({k})\}_{{k}\in\mathbb  Z}
    \|_{l_p({\mathbb Z})}=\Big(\sum_{{ k}\in\mathbb  Z}|\widehat f({k})|^p\Big)^{1/p}, \quad 1\le p<\infty.
 \end{equation}
  As shown in \cite{Abdullayev_Chaichenko_Shidlich_2021},  for any $f\in {\mathcal S}^{p}$ and $p>1$,
    $\|f\|^\ast_{_{\scriptstyle  {\bf M}}}=\|f\|_{_{\scriptstyle {p}}}$.

In the case $p=1$, the similar equality for norms $ \| f\|^\ast_{_{\scriptstyle  {\bf M}}}=
\|f\|_{_{\scriptstyle {1}}}$
 obviously can be obtained if we consider all $M_k(u)=u$, $k\in {\mathbb Z}$, and the set
$\Lambda$ is a set of all sequences of  positive numbers $\lambda=\{\lambda_k\}_{k\in \mathbb{Z}}$ such that  $\|\lambda\|_{_{l_\infty({\mathbb Z})}}=\sup_{k\in \mathbb{Z}}\lambda_k \le 1$.

In particular, if all functions $M_k(t)=\frac {t^2}4$, then the space ${\mathcal S}_{\bf M}$ coincides with the
Lebesgue space $L_2$ of functions $f\in L$ with finite norm
 \[
 \|f\|_{_{\scriptstyle  L_2}}= \bigg(\frac 1{2\pi}\int\limits_{0}^{2\pi}|f(t)|^2
 {\mathrm d}t\bigg)^{1/2}.
 \]

In the spaces ${\mathcal S}^{p}$, we denote the generalized modulus of smoothness of a function $f$  by  $ \omega_\alpha(f,t)_{_{\scriptstyle  p}}$, and  the  best approximation of $f$ by the trigonometric polynomials    ${T}_{n-1}\in {\mathscr T}_{2n-1}$ is denoted by   $E_n (f)_{_{\scriptstyle  p}}$. By $\Omega _\varphi(f, \tau, \mu  , u)_{_{\scriptstyle  p}}$, $u>0$, we  denote the average value of the generalized modulus of smoothness
$\omega _\varphi(f, t)_{_{\scriptstyle  p}}$ of the function  $f$ with the weight $\mu   \in {\mathcal M}(\tau )$. Theorems \ref{Th.3.1},
\ref{Th.3.2} and \ref{Th.3.3} yields the following corollaries:

 %%%%%%%%%%%%%%%%%%%%%%%%%%%%%%%%%%%%%%%%%%%%%%%%%%%%%%%%%%%%%%%%%%%%%%%%%%%%%%%%

\begin{corollary}
       \label{Cor.3.1}
       Assume that $f\in {\mathcal S}^p$, $1\le p<\infty$,   $\varphi\in \Phi$, $\tau>0$,  $\mu  \in  {\mathcal M}(\tau )$
       and $\psi\in \Psi$.     Then for any   $n\in {{\mathbb N}}$ the following inequality is true:
 \begin{equation}\label{Jackson_Type_Ineq_Sp}
    E_n(f)_{_{\scriptstyle {p}}} \le
     \frac {\mu   (\tau ) - \mu   (0)}
    {I_{n,\varphi }(\tau ,\mu   )}  |\psi  (n)|\,
    \Omega _\varphi\Big(f^{\psi}, \tau,\mu   , \frac{\tau }n\Big)_{_{\scriptstyle {p}}},
 \end{equation}%\eqno (6.87)$$,
where   the quantity $I_{n,\varphi}(\tau ,\mu   )$ is defined by $(\ref{I_n,varphi,p})$.
If, in addition, the function $\varphi$ is non-decreasing on
$[0,\tau]$  and condition $(\ref{I_n,varphi,p_Equiv_n})$ holds, then  inequality $(\ref{Jackson_Type_Ineq_Sp})$ can not be improved and therefore,
 \begin{equation}
       \label{Jackson_Type_Exact_1}
       \mathop {\sup\limits _{f\in L^{\psi}{\mathcal S}^p}}\limits _{f\not ={\rm const }}
       \frac {E_n(f)_{_{\scriptstyle {p}}} }
       {\Omega_\varphi (f^{\psi }, \tau, \mu  , \frac{\tau }{n} )_{_{\scriptstyle {p}}} }=
        \frac {\mu  (\tau ) - \mu   (0)}{\int_0^{\tau }\varphi (t) {\mathrm d} \mu     (t)} |\psi  (n)|.
 \end{equation}%(6.90)$$

\end{corollary}

%%%%%%%%%%%%%%%%%%%%%%%%%%%%%%%%%%%%%%%%%%%%%%%%%%%%%%%%%%%%%%%%%%%%%%%%%%%%%%%%%%%%%%%%%%%%%%%%%%%%%%%%%%%%%%%%%%%%%%%%%%%

\begin{corollary} \label{Cor.3.2}
       Assume that  $1\le p<\infty$,  $\psi \in \Psi$, $\tau>0$,   the function
       $\varphi\in \Phi$ is non-decreasing on the interval $[0,\tau]$ and  $\mu   \in {\mathcal M}(\tau )$.
       Then for any  $n\in {\mathbb N}$ and  $N\in \{2n-1,  2n\}$  the following inequalities are true:
\begin{equation}\label{A6.1051}
    \frac{\mu   (\tau ) - \mu   (0)}
   {\int_0^{\tau }\varphi (t) {\mathrm d} \mu     (t)} |\psi (n)|
   \le P_{N} \Big(L^{\psi }(\varphi, \tau, \mu  , n)_{_{\scriptstyle {p}}}  , {\mathcal S}^p \Big)
    \le  \frac {\mu   (\tau ) - \mu  (0)}
   {I_{n,\varphi}(\tau ,\mu   )}    |\psi (n)|,
  \end{equation}%\eqno (6.105)$$
where the quantity $I_{n,\varphi}(\tau ,\mu   )$ is defined by $(\ref{I_n,varphi,p})$, and $P_N$
is any of the widths $b_N$, $d_N$, $\lambda _N$ or $ \pi _N$. If, in addition, condition $(\ref{I_n,varphi,p_Equiv_n})$ holds, then
    \begin{equation}\label{A6.1061} %\[
          P_{N} \Big(L^{\psi }(\varphi, \tau, \mu  , n)_{_{\scriptstyle {p}}}  , {\mathcal S}^p\Big )=
           \frac{\mu   (\tau ) - \mu   (0)}
          {\int_0^{\tau }\varphi(t) {\mathrm d} \mu     (t)} |\psi (n)|.
      \end{equation}%\eqno (6.106)$$
\end{corollary}

%%%%%%%%%%%%%%%%%%%%%%%%%%%%%%%%%%%%%%%%%%%%%%%%%%%%%%%%%%%%%%%%%%%%%%%%%%%%%%%%%%%%%%%%%%%%%%%%%%%%%%%%%%%%%%%%%%%%%%%%%%%

\begin{corollary}
       \label{Cor.3.3}
       Let  $1\le p<\infty$, $\psi \in \Psi$,   the function
       $\varphi\in \Phi$ be non-decreasing on a certain interval $[0,a]$, $a>0$, and
       $\varphi(a)=\sup\{\varphi(t):\, t\in {\mathbb R}\}$.
       Let also  $\tau\in (0,a]$,  the function  $\mu   \in {\mathcal M}(\tau )$
       and for all $\xi>0$ and $0<u\le a $, $\Omega  (u)$ be a fixed continuous monotonically
       increasing function of the variable   $u \ge  0$  such that $\Omega  (0)=0$ and   condition $(\ref{A6.113})$.
      is satisfied.  Then for any  $n\in {\mathbb N}$ and  $N\in \{2n-1,  2n\}$  the following inequalities are true:
\[
             \frac{\mu   (\tau ) - \mu   (0)}{\int_0^{\tau }\varphi(t) {\mathrm d} \mu     (t)}
            |\psi (n)|\,\,\Omega  \Big(\frac {\tau }n\Big)\le
            P_N\Big(L^{\psi }(\varphi, \tau, \mu  , \Omega)_{_{\scriptstyle {p}}}  , {\mathcal S}^p\Big)
            \]
                   \begin{equation}\label{A6.1151} %
                   \le
    \frac {\mu   (\tau ) - \mu  (0)} {I_{n,\varphi}(\tau ,\mu   )}   |\psi (n)|
            \,\,\Omega  \Big(\frac {\tau }n\Big),
      \end{equation}%   \eqno (6.115)$$
where the quantity $I_{n,\varphi}(\tau ,\mu   )$ is defined by $(\ref{I_n,varphi,p})$, and $P_N$
is any of the widths $b_N$, $d_N$, $\lambda _N$ or $ \pi _N$. If, in addition,  condition $(\ref{I_n,varphi,p_Equiv_n})$ holds, then
     \begin{equation}\label{A6.1161} %
      P_{N}(L^{\psi } \Big(\varphi, \tau,  \mu  , \Omega  )_{_{\scriptstyle {p}}}  , {\mathcal S}^p\Big)=
      \frac{\mu   (\tau ) - \mu   (0)}{\int_0^{\tau }\varphi(t) {\mathrm d} \mu     (t)}
            |\psi (n)|\,\,\Omega  \Big(\frac {\tau }n\Big).
      \end{equation}%   \eqno (6.116)$$

\end{corollary}

%%%%%%%%%%%%%%%%%%%%%%%%%%%%%%%%%%%%%%%%%%%%%%%%%%%%%%%%%%%%%%%%%%%%%%%%%%%%%%%%%%%%%%%%%%%%%%%%%%%%%
Let us note that in ${\mathcal S}^p$, the statements similar to Corollaries \ref{Cor.3.1}-\ref{Cor.3.3} were also obtained in \cite{Abdullayev_Serdyuk_Shidlich_2021} and \cite{Serdyuk_2003}. In \cite{Abdullayev_Serdyuk_Shidlich_2021},
instead of the  average values
$\Omega _\varphi(f, \tau, \mu  , u)_{_{\scriptstyle  p}}$ of the generalized moduli of smoothness
$\omega _\varphi(f, t)_{_{\scriptstyle  p}}$, the authors considered the quantities
 \begin{equation}\label{Mean_Value_Gen_Modulus_Sp}
  \Omega _\varphi(f, \tau, \mu ,s , u )_{_{\scriptstyle  p}}  =
   \Big(\frac{1}{\mu    (\tau ) - \mu    (0)}\int %\limits
   _0^u\omega _\varphi^s(f, t)_{_{\scriptstyle  p}}  {\mathrm d}\mu    \Big(\frac {\tau
   t}{u}\Big)\Big)^{1/s}
 \end{equation}
 in the case where $s=p$. In  \cite{Serdyuk_2003}, the quantities $\Omega _\varphi(f, \tau, \mu,p  , u )_{_{\scriptstyle  p}} $
  were considered when the function $\varphi(t)=\varphi_\alpha(t)=2^\frac \alpha 2 (1-\cos{kt})^\frac \alpha2$.  In the case $s=p=1$ and $\psi\in \Psi$, the results of  Corollaries \ref{Cor.3.1}-\ref{Cor.3.3} coincide with the corresponding results of \cite{Abdullayev_Serdyuk_Shidlich_2021}.

 In the case  $p=2$,  $\psi (k) = ({\mathrm i}k)^{-r}$, $r=0, 1,\ldots$ and  $\mu  _1(t) = 1 - \cos t$,  equalities of the kind as  (\ref{Jackson_Type_Ineq_Sp})
 with the quantities $\Omega _{\varphi_\alpha}(f, \tau, \mu  _1,2  , u)_{_{\scriptstyle  2}}$  follow from  the result of Chernykh  \cite{Chernykh_1967_2} when $\alpha=1$, and from results of   Yussef \cite{Yussef_1988}
 when $\alpha=k\in {{\mathbb N}}$ and $n\in {\mathbb N}$. For the weight function $\mu  _2(t)=t$ and
 $\psi (k) = ({\mathrm i}k)^{-r}$, $r=0, 1,\ldots$,    equalities of the kind as  (\ref{Jackson_Type_Ineq_Sp})
 with the quantities $\Omega _{\varphi_\alpha}(f, \tau, \mu  _2,2  , u)_{_{\scriptstyle  2}}$ were obtained by
 Taikov  \cite{Taikov_1976, Taikov_1979}  ($k=1$ or $r\ge 1/2$ and $k \in {{\mathbb N}}$).

The widths of the classes
   \[
   L^{\psi }(\varphi, \tau, \mu,s   , \Omega )_{_{\scriptstyle  p}}=  \Big\{f\in L^{\psi }{\mathcal S}^p:\
    \Omega _\varphi(f^{\psi }, \tau , \mu ,s  ,  u)_{_{\scriptstyle  p}}  \le \Omega  (u),\ 0\le u\le \tau \Big\},
  \]
when $p=s=2,$ $\mu (t)=\mu  _2(t)=t$, $\varphi(t)=\varphi_\alpha(t)=2^\frac \alpha 2 (1-\cos{kt})^\frac \alpha2$,  $\psi (k)=({\mathrm i}k)^{-r}$, for $r\ge 0$ and $\alpha=1$ or $r\ge 1/2$ and $\alpha\in {{\mathbb N}},$  were obtained in  \cite{Taikov_1976, Taikov_1979} (see also \cite[Ch. 4]{Pinkus_1985}), where the existence
 of functions $\Omega$ satisfying condition similar to $(\ref{A6.113})$ was also proved.

%%%%%%%%%%%%%%%%%%%%%%%%%%%%%%%%%%%%%%%%%%%%%%%%%%%%%%%%%%%%%%%%%%%%%%%%%%%%

%%%%%%%% Bibliography %%%%%%%%%%%%%%%%%%%%%%%%%%%%%%%%%%%%%%%%%%%%%%%%%
\footnotesize

       %\textit{Bibliography.}
       % Items must be given in alfabetic order.         You can use usual latex standard, with following font commands:

\end{document}